\documentclass[a4paper, 12pt]{amsart}
\usepackage{amsmath}
\usepackage{amssymb}

\theoremstyle{plain}
\newtheorem{theorem}{\bf Theorem}[section]
\newtheorem{proposition}[theorem]{\bf Proposition}
\newtheorem{lemma}[theorem]{\bf Lemma}

\theoremstyle{definition}

\newtheorem{remark}[theorem]{\bf Remark}

 \DeclareMathOperator{\ord}{ord}

\numberwithin{equation}{section}

\begin{document}

\title[Zero-sum Invariants of  Finite Abelian
Groups ]{Zero-sum Invariants of  Finite Abelian
Groups }

\address{Center for Combinatorics and LPMC-TJKLC, Nankai University, Tianjin,
300071, P.R. China} \email{wdgao@nankai.edu.cn,
}

\address{Department of Mathematics, Brock University, St. Catharines, Ontario,
Canada L2S 3A1} \email{yli@brocku.ca}

\address{College of Science, Civil Aviation University of China, Tianjin, 300300, P.R. China} \email{jtpeng1982@aliyun.com}

\address{Department of Mathematics, Tianjin Polytechnic University, Tianjin, 300387, P. R. China} \email{gqwang1979@aliyun.com}

\author{Weidong Gao, YuanLin Li, Jiangtao Peng and Guoqing Wang}




\begin{abstract}
The purpose of the article is to provide  an unified way to formulate zero-sum invariants. Let $G$ be
a finite additive abelian  group. Let $\mathcal{B} (G)$ denote the
set consisting of all nonempty zero-sum sequences over $G$. For
$\Omega \subset \mathcal{B} (G)$, let $d_{\Omega}(G)$  be the
smallest integer $t$ such that every sequence $S$ over $G$ of length
$|S|\geq t$ has a subsequence in $\Omega$.  We provide some first
results and open problems on  $d_{\Omega}(G)$.
\end{abstract}

\maketitle


\bigskip

\bigskip
\section{Introduction}
\bigskip

Zero-sum theory on abelian groups can be traced back to 1960's and
has been developed rapidly in recent three decades (see
\cite{Ga-Ge06b}, \cite{Ge09a} and \cite{Ge-HK06a}). Many invariants
have been formulated in zero-sum theory and we list some of these
invariants here. Let $G$ be an additive finite abelian group. The
EGZ-constant $\mathsf s(G)$ is the smallest integer $t$ such that
every sequence $S$ over $G$ of length $|S|\geq t$ has a zero-sum
subsequence of length $\exp (G)$. This invariant $\mathsf s(G)$
comes from the so called Erd\H{o}s-Ginzburg-Ziv theorem, which
asserts that $\mathsf s(G) \leq 2|G|-1$ and is regarded as one of
two starting points of zero-sum theory. Another starting point of
zero-sum theory involves the Davenport constant $\mathsf D(G)$,
which is defined as the smallest integer $t$ such that every
sequence $S$ over $G$ of length $|S|\geq t$ has a nonempty zero-sum
subsequence. This invariant $\mathsf D(G)$ was formulated by H.
Davenport in 1965 and his motivation was to study algebraic number
theory. Let $\eta (G)$ be the smallest integer $t$ such that every
sequence $S$ over $G$ of length $|S|\geq t$
 has a zero-sum subsequence of length between 1 and $\exp (G)$, which was first introduced by Emade Boas in 1969 to study $\mathsf D(G)$. Let $\mathsf E(G)$ be the smallest integer $t$ such that every sequence $S$ over $G$ of
 length $|S|\geq t$ has
 a zero-sum subsequence of length $|G|$, which was formulated by the first author in 1996.
  The above invariants on zero-sum sequences have been studied by many authors. To give a
 unified way of
formulating zero-sum invariants, recently,
  Geroldinger, Grynkiewicz and Schmid \cite{GGS} introduced the following concept.  Let $\mathcal L$ be a finite set of positive integers, and let $\mathbf d_{\mathcal L}(G)$ be the smallest integer $t$ such that every
  sequence $S$ over $G$ of length $|S|\geq t$ has a zero-sum subsequence of length in $\mathcal L$. Here we give a more unified way to formulate zero-sum
  invariants.

\medskip
\noindent
 Let $\mathcal B (G)$ denote the
set consisting of all nonempty zero-sum sequences over $G$. For $\Omega \subset \mathcal B (G)$, define $d_{\Omega}(G)$ to be the smallest integer $t$ such that every sequence $S$ over $G$ of length $|S|\geq t$ has a subsequence in $\Omega$. If such $t$ does not exist, we let $d_{\Omega}(G)=\infty$. Now we have
\[
d_{\Omega}(G)=\left \{ \begin{array}{ll}  \mathsf{D}(G), \mbox{ if } \Omega = \mathcal B(G) ; \\
 \eta (G), \mbox{ if } \Omega = \{ S\in \mathcal B(G): 1\leq |S| \leq \exp(G)\}; \\  \mathsf{s}(G), \mbox{ if } \Omega = \{ S\in \mathcal B(G): |S| = \exp(G)\}; \\  \mathsf{E}(G), \mbox{ if } \Omega = \{ S\in \mathcal B(G):  |S|=|G|)\};   \\  \mathbf d_{\mathcal L} (G), \mbox{ if } \Omega =\{S \in \mathcal B(G): |S|\in \mathcal L \}. \end{array} \right.
\]

In this paper, we will provide some first results on
$d_{\Omega}(G)$ and formulate some open problems. The rest of this
paper is organized as follows. In Section 2, we present some
necessary concepts and terminology. In Section 3 we study the way to
present a given integer $t\geq \mathsf D(G)$ by $\Omega \subset
\mathcal B(G)$ with $d_{\Omega}(G)=t$; In the final section we study
the minimal $\Omega \subset \mathcal B(G)$ with respect to
$d_{\Omega}(G)=t$.

\section{Preliminaries}

Let $\mathbb{N}$ denote the set of positive integers,
$\mathbb{N}_0=\mathbb{N}\cup\{0\}$. For a real number $x$, we denote
by $\lfloor x\rfloor$ the largest integer that is less than or equal
to $x$.

Throughout, all abelian groups will be written additively. By the
Fundamental Theorem of Finite Abelian Groups we have
$$G\cong C_{n_1}\oplus\cdots\oplus C_{n_r},$$
where $r=\mathsf r(G)\in \mathbb{N}_0$ is the rank of $G$,
$n_1,\ldots,n_r\in\mathbb{N}$ are integers with $1<n_1|\ldots|n_r$,
moreover, $n_1,\ldots,n_r$ are uniquely determined by $G$, and
$n_r=\text{exp}(G)$ is the $exponent$ of $G$.  Set
$$
M(G)=1+\sum_{i=1}^r(n_i-1).
$$

Let $G_0\subset G$.  For $g_1,\ldots,g_l\in G_0$(repetition
allowed), we call $S=g_1\cdot \ldots \cdot g_l$ a $sequence$ over
$G_0$. We write sequences $S$ in the form
$$S=\prod_{g\in G}g^{\textsf{v}_g(S)}\text{ with }\textsf{v}_g(S)\in \mathbb{N}_0\text{ for all }g\in G_0.$$
We call $\textsf{v}_g(S)$ the $multiplicity$ of $g$ in $S$.

For $S=g_1\cdot \ldots \cdot g_l=\prod_{g\in
G_0}g^{\textsf{v}_g(S)}$, we call
\begin{itemize}
\item $|S|=l=\sum_{g\in G_0}\textsf{v}_g(S)\in \mathbb{N}_0$ the $length$ of $S$.

\item $\sigma(S)=\sum\limits_{i=1}\limits^{l}g_i=\sum_{g\in G_0}\textsf{v}_g(S)g\in G_0$ the $sum$ of $S$.

\item $S$ is a $zero$-$sum$ $sequence$ if $\sigma(S)=0$.

\item $S$ is a $short$ $zero$-$sum$ $sequence$ if it is a zero-sum sequence of length $|S|\in[1,\text{exp}(G)]$

\end{itemize}

 We denote by $\mathcal B (G_0)$ the set  of all nonempty zero-sum sequences over $G_0$,   by  $\mathcal A(G_0)$  the set of all minimal
zero-sum sequences over $G_0$.

We say two sequences $S$ and $T$ over $G$ have the same form if and
only if $\textsf{v}_g(S)=\textsf{v}_g(T)$ for every $g\in G$.

\section{Representing the same invariant $t$ by different $\Omega$
with $d_{\Omega}(G)=t$}

We first state some basic properties on $d_{\Omega}(G)$.

\medskip
\begin{proposition} \label{prop2.2}  Let $\Omega,\Omega' \subset \mathcal B(G)$.
\begin{enumerate}

\item $d_{\Omega}(G)< \infty$ if and only if for every element $g\in G$, $g^{k\ord(g)} \in \Omega$ for some positive integer $k=k(g)$

 \item
 If $\Omega \subset \Omega'$ then $d_{\Omega'}(G)\leq d_{\Omega}(G).$

 \item If $S_1$ is a proper subsequence of $S$ and both $S$ and $S_1$
 belong to $\Omega$, then $d_{\Omega \setminus \{S_1\}}(G)=
 d_{\Omega}(G)$.

\item For every positive integer $t \in [\mathsf D(G), \infty ),$ there is an $\Omega \subset \mathcal B(G)$ such that $t=d_{\Omega}(G).$

\end{enumerate}

\end{proposition}

\begin{proof}

1. Necessity. Let $t=d_{\Omega}(G)< \infty$. For every $g\in G$, the
sequence $g^t$ has a nonempty zero-sum subsequence $S$ in $\Omega$.
Since $S$ is zero-sum, we infer that it has the form $g^{k\ord(g)}$
for some positive integer $k$. This proves the necessity.

Sufficiency. If for every $g\in G$ there is a positive integer $k_g$
such that $g^{k_g\ord (g)}\in \Omega$, then clearly
$d_{\Omega}(G)\leq 1+\sum_{g\in G}(k_g\ord (g)-1)$ completing the
proof of Conclusion 1.

\medskip
2. The result holds obviously.

\medskip
3. The result follows from 2.

\medskip
4. Let $G^*=G\setminus \{0\}$. For each positive integer $t\geq
\mathsf D(G)$.   Let $\Omega=\{0^{t-\mathsf D(G)+1}\} \cup \mathcal
B(G^*)$. It is easy to see that $t=d_{\Omega}(G)$, completing the
proof.
\end{proof}

\bigskip
 Although Conclusion 4 of the above proposition asserts that for every positive $t \geq  \mathsf D(G)$, there is an $\Omega \subset \mathcal B(G)$ such that $t=d_{\Omega}(G)$, it does not give us much information on the invariant $t$.
 For some classical invariants $t$, we hope to
find some special $\Omega \in \mathcal B(G)$ with $d_{\Omega} (G)=t$
to help us understand $t$ better. We say a sequence $S$ over $G$ a
{\sl weak-regular} sequence if $\mathsf{v}_g(S)\leq \ord (g)$ for every $g\in
G$. We say $\Omega \subset \mathcal B(G)$ is {\sl weak-regular} if
very sequence $S\in \Omega$ is weak-regular.  Define $Vol(G)$ to be
the set of all positive integer $t\in [\mathsf D(G), 1+\sum_{g\in G}
(\ord(g)-1)]$ such that $t=d_{\Omega}(G)$ for some weak-regular
$\Omega \subset \mathcal B(G)$.

\begin{remark} A sequence over $G$ is called a regular sequence if $|S_H|\leq
|H|-1$ for every proper subgroup $H$ of $G$, where $S_H$ denote the
subsequence consisting of all terms of $S$ in $H$. The concept of
regular sequences was introduced by Gao, Han and Zhang \cite{GHZ}
quite recently.
\end{remark}

\medskip
{\bf Question 1.} Does $Vol (G)=[\mathsf D(G), 1+\sum_{g\in G} (\ord(g)-1)]$ hold for any finite abelian group $G$?

\medskip
Let $\Omega \subset \mathcal B(G)$, we say a sequence $S$ over $G$
is $\Omega$-{\sl free} if $S$ has no subsequence in $\Omega$.

\begin{lemma} \label{lemma2.2} Let $\Omega \in \mathcal B(G)$ with $d_{\Omega}(G)=t
\in [\mathsf D(G), 1+\sum_{g\in G} (\ord(g)-1)]$. If there is an
 $\Omega$-free sequence $T$ over $G$ of length $|T|=t-1$ such that
 $\mathsf v_g(T)\leq \ord(g)-1$ for every $g\in G$, then $d_{\Omega}(G)=t \in Vol (G)$.
\end{lemma}

\begin{proof} If every sequence in $\Omega$ is weak regular, then by
the definition of $Vol(G)$ we know that $t\in Vol(G)$.
Otherwise, take an $S\in \Omega$ such that $\mathsf v_g(S)\geq \ord (g)+1$ for some $g\in G$.
Let $\Omega_1=(\Omega \setminus \{S\})\cup \{g^{\ord (g)}\}$. It is
easy to see that $d_{\Omega_1}(G) \leq d_{\Omega}(G)=t$. On the
other hand, $T$ is also $\Omega_1$-free. Therefore,
$d_{\Omega_1}(G)\geq t$. Hence,  $d_{\Omega_1}(G)=t$. Continuing the
same process above we finally get an $\Omega_k$ which is weak regular such
that  $d_{\Omega_k}(G)=t \in Vol (G)$.
\end{proof}

\begin{proposition} \label{prop1.1} For every finite  abelian group $G$ we have
$$
\mathsf D(G) \in Vol(G),  \mathsf D(G)+1 \in Vol(G)  \mbox{ and }
\eta (G) \in Vol (G).
$$

\end{proposition}

\begin{proof} Let $\Omega=\mathcal B(G)$. Then,  $d_{\Omega}(G)=\mathsf D(G)$. Let $T$ be a  zero-sum-free sequence (or equivalently, $\Omega$-free) over $G$ of length $|T|=\mathsf D(G)-1=d_{\Omega}(G)-1$.  Clearly,
 $T$ is weak regular and $\mathsf D(G)=d_{\Omega}(G) \in Vol (G)$ by Lemma \ref{lemma2.2}.

To prove $\mathsf D(G)+1\in Vol (G)$, let $\Omega =\{ S\in \mathcal
B(G), |S|<\mathsf D(G )\} \cup \{g^{\ord(G)}, g\in G\}$. It is easy
to see that $d_{\Omega}(G)=\mathsf D(G)+1$. Let $\Omega' =\{ S\in
\mathcal A(G), |S|<\mathsf D(G )\} \cup \{g^{\ord(G)}, g\in G\}$.
Clearly, $d_{\Omega'}(G)=d_{\Omega}(G)=\mathsf D(G)+1$ and $\mathsf
D(G)+1\in Vol (G)$ follows from $\Omega'$ is weak-regular.

To prove $\eta(G)\in Vol (G)$, let $\Omega=\{S\in \mathcal B(G),
|S|\leq \exp(G)\}$. It is easy to see that $d_{\Omega}(G)=\eta (G)$.
Let $T$ be any $\Omega$-free sequence of the maximal length
$|T|=\eta(G)-1$. Clearly, $\mathsf v_g(T)\leq \ord(g)-1$ for every $g\in G$ and hence $\eta(G)\in
Vol (G)$ by Lemma \ref{lemma2.2} again.
\end{proof}

\section{The minimal $\Omega$ with respect to $d_{\Omega}(G)=t$ for some given $t$}

For a given $t\geq \mathsf{D}(G)$ and an $\Omega \subset \mathcal B(G)$ with $d_{\Omega}(G)=t$. We say that $\Omega $ is minimal respect to $d_{\Omega}(G)=t$ if $d_{\Omega'}(G)>t$ for every proper subset
$\Omega' \subsetneq \Omega$.

By Proposition \ref{prop2.2} (3)  we have the following.
\begin{proposition} \label{prop1.2} Let $t\geq \mathsf{D}(G)$ be an integer, and let $\Omega \subset \mathcal B(G)$ be minimal with respect to $d_{\Omega}(G)=t$. Then, for every $B, B'\in \Omega$
with $B\neq B'$ we have $B' \nsubseteq B$.
\end{proposition}

\medskip
{\bf Question 2.} Let $\Omega \in \mathcal B(G)$ be minimal with respect to $d_{\Omega}(G)=\mathsf{D}(G)$. What can be said about $\Omega$?

\medskip
For     Question 2, $d_{\Omega}(G)=\mathsf{D}(G)$ if
$\Omega=\mathcal A(G)$, but we can not say that $\mathcal A(G)$ is
smallest. Let $p$ be a prime. Let $G=C_p$.  Let $\Omega_1$ be the
set consisting of all minimal zero-sum sequences over $G$ with index
1. If the Lemke-Kleitman's conjecture is true for $p$, i.e., every
sequence over $G$ of length $p$ has a subsequence with index 1, then
$\Omega_1$ is a proper subset of $\mathcal A(G)$  with
$d_{\Omega_1}(G)=\mathsf D(G)=p$.  So  if $n=p\geq 7$ is a prime and
the Lemke-Kleitman's conjecture mentioned above is true for $p$,
then $\Omega =\mathcal A (G)$ is not minimal with respect to
$d_{\Omega}(G)=\mathsf D(G)=p$. In fact, except for some small
primes $p$, we don't know any minimal $\Omega \subset \mathcal B(G)$
with respect to $d_{\Omega}(G)=\mathsf D(G)=p$. For sufficiently
large $t$, we can find minimal $\Omega $ with respect to
$d_{\Omega}(G)=t$.

\begin{proposition} \label{lemma1} Let $G$ be a finite abelian group and $t\geq 1+\sum_{g\in G}(\ord(g)-1)$ be a integer, and let $\Omega =\{g^{\ord(g)}, g\in G \}\cup \{0^{t-\sum_{g\in
G}(\ord(g)-1)}\}$.   Then $\Omega$ is minimal with respect to
$d_{\Omega}(G)=t$.

\end{proposition}

\begin{proof} Clearly.

\end{proof}

\medskip
 For a
given $t\in [ \mathsf D(G), \sum_{g\in G}(\ord(g)-1)]$ not too
close to $\sum_{g\in G}(\ord(g)-1)$, it could be very difficult to
find a minimal $\Omega \subset \mathcal B(G)$ with respect to
$d_{\Omega}(G)=t.$ We provide one more example here:

\medskip
\begin{proposition} \label{prop3.2}
Let $n\geq 5$ be an odd integer. Then $\Omega =\{T\in \mathcal
B(C_n)| \mbox{ } |T|=n\}$ is not minimal with respect to
$d_{\Omega}(G)=2n-1$.
\end{proposition}

\begin{proof} Let $S_0=\prod_{i=0}^{n-1}i$. Since $n$ is odd, we have $S_0\in \Omega$. Let $\Omega_1=\Omega \setminus \{S_0\}$. Clearly, $d_{\Omega_1}(G)\geq d_{\Omega}(G)=2n-1$. So, it suffices to prove $d_{\Omega_1}(G)\leq 2n-1.$ Let $S$ be an arbitrary sequence over $G$ with length
$|S|=2n-1$. We need to show $S$ contains a zero-sum subsequence in $\Omega_1$. Since $d_{\Omega}(G)=2n-1$, we may assume that all zero-sum subsequences of $S$ with length $n$ is of the same form as $S_0=\prod_{i=0}^{n-1}i$. It follows that $j$ occurs exactly once in $S$ for some $j\in [0,n-1]$ . Consider the sequence $S_1=S\setminus \{j\}$. Then, $|S_1|=2n-2$ and $S_1$ contains $n-1>2$ distinct elements. By recalling the well known fact that a sequence over $C_n$ with length $2n-2$ having no zero-sum subsequence of length $n$ if and only if the sequence consists of two distinct elements with each appearing $n-1$ times, we derive that $S_1$ and hence $S$ has a zero-sum subsequence in $\Omega_1$.
\end{proof}

\medskip
 We say a zero-sum sequence $S$ is {\sl essential}
with respect to some $t\geq \mathsf{D}(G)$ if every $\Omega \in
B(G)$ with $d_{\Omega}(G)=t$ contains $S$. For examples, for $G=C_n$
and $t=n$, the zero-sum sequence $g(-g)$ is essential with respect
to $n$ for every generator $g$ of $C_n$. For any finite abelian
group $G$, every minimal zero-sum sequence over $G$ of length
$\mathsf{D}(G)$ is essential with respect to $\mathsf{D}(G)$.

\medskip
A natural research problem is  to determine the smallest integer $t$ such that there is no essential zero-sum sequence with respect to $t$, denote it by $q(G)$.

\medskip
Let $q'(G)$ be the smallest integer $t$ such that, every sequence $S$ over $G$ of length $|S|\geq t$ has two nonempty zero-sum subsequences with different forms.

{\bf Question 3.} Does $q(G)=q'(G)$ for any finite abelian group $G$?

\medskip
Relating to $q'(G)$ we present two invariants below.

Let $disc(G)$ denote the smallest integer $t$ such that every
sequence $S$ over $G$ of length $|S|\geq t$ has two nonempty
zero-sum subsequences with different lengths. The invariants
$disc(G)$ was formulated by Girard \cite{Gir} and has been studied
recently in \cite{GZZ} and \cite{GLZZ}.

Let $\mathsf {D}_2(G)$ denote the smallest integer $t$ such that every
sequence over $G$ has two disjoint nonempty zero-sum subsequences.
Clearly,
$$
 q'(G)  \leq disc(G)\leq \mathsf {D}_2(G)
$$
for every finite abelian group $G$.

$\mathsf D_2(G)$ was first introduced by H.~Halter-Koch in \cite{Koch} and
was studied recently by Plagne and Schmid in \cite{PS1}.

We close this section with the following theorem.

\begin{theorem} \label{theorem1.1} If $G$ is a finite abelian group
then $q(G)\leq \mathsf D_2(G)$.
\end{theorem}

\begin{proof}

Let $G^*=G\setminus \{0\}$. Let $k\geq \mathsf D_2(G)$ be a integer
and we only need to show that there exist two disjoint subsets
$\Omega, \Omega' \subset \mathcal B(G)$ such that
$d_{\Omega}(G)=k=d_{\Omega'}(G)$. Let $\Omega=\{0^{k-\mathsf
D(G)+1}\}\cup  \mathcal A(G^*)$ and let $\Omega'=\{0^{k-\mathsf
D_2(G)+1}\}\cup \{\mathcal B(G^*)\setminus \mathcal A(G^*)\}$.

Note that a minimal zero-sum sequence over $G$ of length $\mathsf
D(G)$ has no two disjoint nonempty zero-sum subsequences, so we obtain
that $\mathsf D_2(G)>\mathsf D(G)$. Thus,
$$
\Omega \cap \Omega'=\emptyset.
$$

On the other hand, it is easy to see that
$d_{\Omega}(G)=k=d_{\Omega'}(G)$ and we are done.

\end{proof}

\bigskip
\noindent {\bf Acknowledgements.} This work was supported by the
PCSIRT Project of the Ministry of Science and Technology, the
National Science Foundation of China, and a Discovery Grant from the
Natural Science and Engineering Research Council of Canada.


\providecommand{\bysame}{\leavevmode\hbox to3em{\hrulefill}\thinspace}
\providecommand{\MR}{\relax\ifhmode\unskip\space\fi MR }
\providecommand{\MRhref}[2]{%
  \href{http://www.ams.org/mathscinet-getitem?mr=#1}{#2}
}
\providecommand{\href}[2]{#2}

\end{document}